\documentclass[12pt]{amsart}
\usepackage{amssymb, amsfonts, amsthm}
\usepackage{graphicx}
\usepackage{float}
\newtheorem{theorem}{Theorem}
\newtheorem{lemma}[equation]{Lemma}
\newtheorem{corollary}{Corollary}

\theoremstyle{definition}

\newcounter{minutes}
\divide\time by 60
\newcounter{hours}
\multiply\time by 60 \addtocounter{minutes}{-\time}

\begin{document}

\title[Estimates for the hyperbolic and quasihyperbolic metrics in hyperbolic regions]
{Estimates for the hyperbolic and quasihyperbolic metrics in hyperbolic regions}

\def\thefootnote{}
\footnotetext{ \texttt{\tiny File:~\jobname .tex,
          printed: \number\day-\number\month-\number\year,
          \thehours.\ifnum\theminutes<10{0}\fi\theminutes}
} \makeatletter\def\thefootnote{\@arabic\c@footnote}\makeatother

\author{Swadesh Kumar Sahoo}
\address{Swadesh Kumar Sahoo, Discipline of Mathematics,
Indian Institute of Technology Indore,
Indore 452 017, India}
\email{swadesh@iiti.ac.in}


\begin{abstract}
In this paper we consider ordinary derivative of universal covering mappings $f$ of hyperbolic regions $D$ in the complex plane. 
We obtain sharp bounds for the ratio $|f'(z)|/{\rm dist}(f(z),\partial f(D))$ in terms of the 
hyperbolic density in simply connection domains. In arbitrary domains, we find a
necessary and sufficient condition for an upper bound for the quantity 
$|f'(z)|/{\rm dist}(f(z),\partial f(D))$ to hold in terms of the hyperbolic density.
As an application of the above results, it is observed that the bounds for the quantity 
of the above type are closely connected with similar bounds for $|f''(z)/f'(z)|$.
\end{abstract}

\maketitle
\pagestyle{myheadings}
\markboth{S. K. Sahoo}{hyperbolic regions}

\thispagestyle{empty}

\section{Introduction and preliminaries}
In this paper, we mean {\em regions} which are open and connected sets.
In view of the {\em Riemann mapping} and {\em uniformization theorems} in the classical complex analysis,
the unit disk $\mathbb{D}=\{z\in\mathbb{C}:\,|z|<1\}$ is usually considered
as a standard region. Both the theorems are connected with hyperbolic regions in the complex plane $\mathbb{C}$. 
Suppose that $D$ is a {\em hyperbolic region} in the complex plane; that is, $\mathbb{C}\setminus D$ contains
at least two points carries a Poincar\'{e} metric. Since the universal covering surface of such a region is conformally 
equivalent to the unit disk $\mathbb{D}$, there is a locally univalent, analytic covering map $f$ mapping $\mathbb{D}$
onto $D$ satisfying
\begin{equation}\label{sec1-eq0}
\lambda_D(f(z))|f'(z)|=\lambda_{\mathbb{D}}(z)=\frac{1}{1-|z|^2},\quad z\in\mathbb{D}.
\end{equation}
We say $\lambda_D(w)$, $w\in D$, {\em the hyperbolic or Poincar\'{e} metric} of $D$. 
We also sometimes call this quantity {\em the hyperbolic density}.
The {\em hyperbolic distance} between two points $w_1$ and $w_2$ in $D$ is defined by the quantity
$$h_D(w_1,w_2)=\inf\int_\gamma\lambda_D(w)\,|dw|
$$
where the infimum is taken over all rectifiable paths connecting $w_1$ and $w_2$ in $D$.
As there is always a question as to the normalization of $\lambda_D(z)$, we consider the definition
to have its curvature $-4$. Several properties of 
this metric can be found, for instance, in \cite{KL07,Leh87,Pom92}. 
Note that this definition is independent of the choice of $f$ in the sense that it continues to hold if $f$ is replaced by
$f\circ h$, where $h$ is a M\"{o}bius transformation of the unit disk onto itself. If $D$ is simply connected, then $f$
is a conformal (analytic and univalent) mapping of $\mathbb{D}$ onto $D$. If $w\in D$ 
and $f:\,\mathbb{D}\to D$ is a conformal mapping with $f(0)=w$
then we have the useful relation
\begin{equation}\label{sec1-eq1}
\lambda_D(w)=\frac{1}{|f'(0)|}.
\end{equation}
At this point the only fact we require about the hyperbolic metric is its conformal invariance. If
$f$ is a conformal mapping of a region $G$ onto $D$ then
\begin{equation}\label{sec1-eq2}
\lambda_D(f(z))|f'(z)|=\lambda_G(z),\quad z\in G. 
\end{equation}
This follows easily from (\ref{sec1-eq0}) and (\ref{sec1-eq1}).

If $\delta_D(z)$ denotes the Euclidean distance from $z$ to the boundary, 
$\partial D$, of $D$ then the {\em quasihyperbolic metric} on $D$ is $|dz|/\delta_D(z)$. 
The {\em quasihyperbolic distance} \cite{GP76} between pair of points $z_1,z_2\in D$ is
defined by 
$$k_D(z_1,z_2)=\inf\int_\gamma \frac{|dz|}{\delta_D(z)}
$$
where the infimum is chosen over all rectifiable paths $\gamma\in D$ joining $z_1$ and $z_2$.
The sharp relations between the hyperbolic and quasihyperbolic metrics, namely,
\begin{equation}\label{sec1-eq3}
\frac{1}{4}\frac{1}{\delta_D(z)}\le \lambda_D(z)\le \frac{1}{\delta_D(z)},\quad z\in D 
\end{equation}
is well-known. The right-hand inequality holds for any hyperbolic region $D$ and follows from the monotonicity
of the hyperbolic metric; if $D_1\subset D_2$ then $\lambda_{D_1}\le \lambda_{D_2}$, $z\in D_1$. The left-hand
inequality holds for any simply connected region and is implied by (actually equivalent to) the Koebe $1/4$-Theorem.
Both the inequalities are sharp (see for instance \cite{Pom92}).  

If $f$ is analytic and locally univalent then the pre-Schwarzian and Schwarzian derivatives of $f$ are respectively
defined to be
$$T_f(z)=\frac{f''(z)}{f'(z)}
~~\mbox{ and }~~
S_f(z)=T_f'(z)-\frac{1}{2}(T_f(z))^2.
$$ 
Among their several properties, whenever the compositions make sense with suitable choice of $g$, 
they exhibit the following transformation laws (see \cite{Leh77,Osg82}):
\begin{equation}\label{composition}
T_{f\circ g}(z)=(T_f\circ g)g'(z)+T_g(z) 
~~\mbox{ and }~~
S_{f\circ g}(z)=(S_f\circ g)(g'(z))^2+S_g(z).
\end{equation}
Using these transformations, Lehto \cite{Leh77} and Osgood \cite{Osg82} respectively proved that 
$$|S_f(z)|\le 12\,\lambda_D(z)^2
~~\mbox{ and } 
|T_f(z)|\le 8\,\lambda_D(z)
$$ 
whenever $f$ is analytic and univalent in a simply connected region $D$. In contrast, the situation in
arbitrary hyperbolic regions is much different. Indeed, in one hand Beardon and Gehring \cite{BG80} have shown
$|S_f(z)|\le 12\,\lambda_D(z)^2$, and on the other hand Osgood \cite{Osg82} has obtained a similar property 
equivalent to an inequality involving the hyperbolic and quasihyperbolic densities. These properties are again sharp.
In terms of the quasihyperbolic density, Gehring \cite{Geh77} has proved the 
sharp inequality $|S_f(z)|\le 6/\delta_D(z)^2$,
whereas Osgood \cite{Osg82} has proved the sharp inequality $|T_f(z)|\le 4/\delta_D(z)$ for arbitrary proper subregions $D$
of the complex plane (see also \cite[pp.~395]{MS79}).

In the sequel, our aim in this paper is to search for functions $f$ satisfying the relation(s) (\ref{composition}) 
and study their similar properties in terms of the hyperbolic and quasihyperbolic metrics. 
Indeed, we found that the ordinary derivative function $f'(z)$ satisfies 
one of these composition properties and leads results of the following type:\\

\noindent
{\bf Theorem~A.} If $D\subset \mathbb{C}$ is a hyperbolic region and $f(z)$ is a 
universal covering mapping in $D$, then 
$$\frac{|f'(z)|}{\delta_{f(D)}(f(z))}\le c\,w_D(z), \quad z\in D.
$$ 
for some constant $c>0$.

Here, the notation $w_D(z)$ is used for the density function which is either the hyperbolic density $\lambda_D(z)$ 
or the quasihyperbolic density $1/\delta_D(z)$. Note that a similar result of type Theorem~A
has recently been obtained in \cite{KVZ14} with respect to the quasihyperbolic density as
an upper bound.

\section{Simply connected hyperbolic regions}
The famous {\em Koebe distortion theorem} plays a crucial role to prove most of our theorems in this paper.\\
\noindent 
{\bf Theorem B}. If $f$ maps $\mathbb{D}$ conformally into $\mathbb{C}$ then
$$\frac{1}{4}(1-|z|^2)|f'(z)|\le {\rm dist}(f(z),\partial{f(\mathbb{D})})\le (1-|z|^2)|f'(z)|
$$
for $z\in \mathbb{D}$.

We now prove our first result by making use of the relations (\ref{sec1-eq1}) and the Koebe distortion theorem.
\begin{theorem}\label{sec2-thm1}
If $D\subset \mathbb{C}$ is a simply connected region and $f(z)$ is a conformal mapping in $D$,
then 
$$\lambda_D(z)\le \frac{|f'(z)|}{\delta_{f(D)}(f(z))}\le 4\,\lambda_D(z)
$$ 
for all $z\in D$. The inequalities are sharp.
\end{theorem}
\begin{proof}
Choose a conformal mapping $g$ of $\mathbb{D}$ onto $D$ with $g(0)=z$.
Then $f\circ g$ is conformal and satisfies the relation
$$|f'(g(\zeta))|\,|g'(\zeta)|= |(f\circ g)'(\zeta)|
$$
for all $\zeta\in \mathbb{D}$. In particular, when $\zeta=0$ we get
$$|f'(z)|\,|g'(0)|= |(f\circ g)'(0)|
$$
for all $z\in D$. Since $g$ and $f\circ g$ are conformal in $\mathbb{D}$, it follows from (\ref{sec1-eq1}) and
the Koebe distortion theorem that 
$$\frac{1}{|g'(0)|}\delta_{f(D)}(f(z))\le |f'(z)|\le \frac{4}{|g'(0)|}\delta_{f(D)}(f(z))=4\,\lambda_D(z)\delta_{f(D)}(f(z))
$$
for all $z\in D$.

For the sharpness in the right hand estimate, we consider the Koebe function $k(z)=z/(1-z)^2$, $z\in \mathbb{D}$. Then
we see that for $z=x<1$
$$k'(x)=\frac{1+x}{(1-x)^3} ~~\mbox{ and }~~
\delta_{k(\mathbb{D})}(k(x))=|k(x)+(1/4)|=\frac{(1+x)^2}{4(1-x)^2}.
$$
Hence,
$$\frac{|k'(x)|}{\delta_{k(D)}(k(x))}=\frac{4}{1-x^2}=4\,\lambda_{\mathbb{D}}(x).
$$

On the other hand, we consider the identity function $f(z)=z$ in the unit disk $\mathbb{D}$. Clearly
$$\frac{|f'(0)|}{\delta_{\mathbb{D}}(f(0))}=1=\lambda_{\mathbb{D}}(0).
$$
The assertion follows.
\end{proof}

As a consequence of Theorem~\ref{sec2-thm1}, we get
\begin{corollary}
If $D\subset \mathbb{C}$ is a simply connected region and $f(z)$ is a conformal mapping in $D$ then
$$h_D(z_1,z_2)\le k_{f(D)}(f(z_1),f(z_2))\le 4\,h_D(z_1,z_2)
$$
for all $z_1,z_2\in D$. The inequalities are best possible.
\end{corollary}
\begin{proof}
Let $\gamma$ be a hyperbolic geodesic segment joining $z_1$ and $z_2$ in $D$, and $\gamma'=f(\gamma)$. Then by
the definition of the quasihyperbolic distance and Theorem~\ref{sec2-thm1} we have
$$k_{f(D)}(f(z_1),f(z_2))\le \int_{\gamma'}\frac{|f'(z)|\,|dz|}{\delta_{f(D)}(f(z))}
\le 4\,\int_{\gamma}\lambda_D(z)|dz|=4\,h_D(z_1,z_2).
$$
A similar argument yields the lower bound.
\end{proof}

As an application of Theorem~\ref{sec2-thm1} and a result of Lehto \cite[pp.~53]{Leh87},
we see that the derivative of the Schwarzian derivative also satisfies a similar property.
\begin{corollary}
Let $\varphi$ be a conformal mapping in a simply connected region $D$ of $\mathbb{C}$. Then 
there is a meromorphic univalent function $f$ in $D$ such that 
$$\frac{|S_f'(z)|}{\delta_{S_f(D)}(S_f(z))}\le 4\,\lambda_D(z)
$$ 
for all $z\in D$. Here, the hyperbolic metric $\lambda_D(z)$ is taken associated with the conformal mapping $\varphi$.
\end{corollary}

\begin{theorem}
If $f(z)$ is conformal in a simply connected hyperbolic region $D$, then there exists an analytic function 
$g(z)$ in $D$ such that 
$$|g'(z)|\le 8\,\lambda_D(z)
$$
for all $z\in D$. Here, the hyperbolic metric $\lambda_D(z)$ is taken associated with the 
conformal mapping $f(z)$.
\end{theorem}
\begin{proof}
We shall use the well-known fact: {\em if $f(z)$ is analytic and nonzero in a simply connected domain
D, then there exists a function $g(z)$, analytic in $D$, such that $\exp[g(z)] = f(z)$ for all $z\in D$.}

Since $f(z)$ is conformal in $D$, $f'(z)$ is non-vanishing there in. Therefore, by the above quoted fact
there exists an analytic function $g(z)$ in $D$ such that $\exp[g(z)]=f'(z)$ for all $z\in D$. Taking derivative
on both the sides, we obtain
$$g'(z)=\frac{f''(z)}{f'(z)},\quad z\in D.
$$ 
The conclusion follows from \cite[Theorem~1]{Osg82}.
\end{proof}

\section{Multiply connected hyperbolic regions}
A statement of the type stated in Theorem~\ref{sec2-thm1} does not hold in an arbitrary region.
For instance, if we choose $\mathbb{D}^*:=\mathbb{D}\setminus\{0\}$ and $f(z)=1/z$ then as
obtained in \cite{Ahl73,KL07} and well-known we have
$$\lambda_{\mathbb{D}^*}(z)=\frac{1}{2}\frac{1}{|z|\log ({1}/{|z|})}.
$$
Note that in \cite{Ahl73} the curvature of $\lambda_{\mathbb{D}^*}$ was $-1$.
It is then easy to compute that 
$$\sup_{z\in\mathbb{D}^*}\frac{\lambda_{\mathbb{D}^*}(z)^{-1}|f'(z)|}{\delta_{f(D)}(f(z))}
=\sup_{z\in\mathbb{D}^*}\frac{2\log(1/|z|)}{1-|z|}
=\infty 
$$
as can be seen when $|z|\to 0$.

On a more positive note we here present a simple characterization for a result 
like Theorem~\ref{sec2-thm1} to hold in a multiply connected region. 
To do this, we first collect a result of type Theorem~\ref{sec2-thm1} but 
the upper bound with respect to the quasihyperbolic density. 
In \cite{KVZ14}, the authors have recently obtained a similar result (see Proposition 1.6) 
using the Koebe one-quarter theorem. It is appropriate to recall the result here.
\begin{lemma}\label{sec3-thm1}
If $D$ is a proper subregion of the complex plane and if $f$ is a conformal mapping in $D$ then 
$$\frac{1}{4\delta_D(z)}\le \frac{|f'(z)|}{\delta_{f(D)}(f(z))}\le \frac{4}{\delta_D(z)}
$$ 
for all $z\in D$. The bounds are sharp.
\end{lemma}
Note that one can also prove Lemma~\ref{sec3-thm1} by using the similar technique as in
the proof of \cite[Proposition 1.6]{KVZ14} with the help of Theorem~B (the Koebe distortion theorem)
and considering the simple conformal mapping $g(z)=f(\delta_0z+z_0)$ in the unit disk $\mathbb{D}$.
%
%

In the beginning of this section we proved that Theorem~\ref{sec2-thm1} does 
not hold in any arbitrary domains. Here we present a necessary and 
sufficient condition under which the quantity 
$|f'(z)|/\delta_{f(D)}(f(z))$ always has an upper bound in terms of the hyperbolic
density of arbitrary hyperbolic domains.
 
\begin{theorem}\label{sec3-thm2}
Let $D\subset \mathbb{C}$ have at least two boundary points. There exists a constant $b$ such that 
$$\frac{|f'(z)|}{\delta_{f(D)}(f(z))}\le b\,\lambda_D(z),\quad z\in D,
$$
for all universal covering mappings $f$ in $D$ if and only if there exists a constant $c>0$ such that
$$\lambda_D(z)\ge \frac{c}{\delta_D(z)}
$$ 
for all $z\in D$.
\end{theorem}
\begin{proof}
Suppose that $\lambda_D(z)\ge {c}/{\delta_D(z)}$ for some constant $c>0$. 
If $f$ is a universal covering mapping in $D$, by Theorem~\ref{sec3-thm1} we get
$$\frac{|f'(z)|}{\delta_{f(D)}(f(z))}\le \frac{4}{\delta_D(z)}\le \frac{4}{c}\lambda_D(z).
$$
Thus, ${|f'(z)|}/{\delta_{f(D)}(f(z))}\le b\,\lambda_D(z)$ with $b=4/c$.

Conversely, we assume that ${|f'(z)|}/{\delta_{f(D)}(f(z))}\le b\,\lambda_D(z)$ for some constant $b>0$ 
and all universal covering mappings $f$ in $D$.
Let $z_0\in D$ be arbitrary. Choose $\zeta_0\in\partial D$ such that $\delta_D(z_0)=|z_0-\zeta_0|$. Now, in particular,
for the choice $f(z)=\delta_{f(D)}(f(z_0)){\rm Log}\,(z-\zeta_0)$ defined on $D$ we have
$$|f'(z_0)|=\frac{\delta_{f(D)}(f(z_0))}{|z_0-\zeta_0|}=\frac{1}{\delta_D(z_0)}\delta_{f(D)}(f(z_0)).
$$
Hence
$$\frac{|f'(z_0)|}{\delta_{f(D)}(f(z_0))}=\frac{1}{\delta_D(z_0)}\le b\,\lambda_D(z_0)
$$
Since $z_0$ was arbitrary, it leads to $\lambda_D(z)\ge c/\delta_D(z)$ with $c=1/b$.
\end{proof}
Note that the value of the constant $c$ in Theorem~\ref{sec3-thm2} is correct when $D$ is chosen to be
simply connected. This follows from Theorem~\ref{sec2-thm1} and the fact (\ref{sec1-eq3}).

As an application of Theorem~\ref{sec3-thm2} and \cite[Theorem~2]{Osg82}, we obtain
\begin{corollary}\label{sec3-cor1}
Let $D\subset \mathbb{C}$ have at least two boundary points and $f$ be a universal covering mapping in $D$. 
There exists a constant $a$ such that 
$$\frac{|f'(z)|}{\delta_{f(D)}(f(z))}\le a\,\lambda_D(z)
$$
if and only if there exists a constant $b>0$ such that
$$\Big|\frac{f''(z)}{f'(z)}\Big|\le b\,\lambda_D(z)
$$ 
for all $z\in D$.
\end{corollary}
Corollary~\ref{sec3-cor1} shows that whenever we have an upper bound for ${|f'(z)|}/{\delta_{f(D)}(f(z))}$ in 
terms of the hyperbolic density, 
we should have a related upper bound for $|f''(z)/f'(z)|$ as well and vice versa. We also note that, for the necessary part,
the resultant inequality holds with $b=4a$. However, for the sufficient part both the constants are identical.

\section{Uniformly perfect hyperbolic regions}
In this section, we discuss some applications of results proved in Section~3.
We refer \cite[Chapter~15]{KL07} for the definition of a uniformly perfect hyperbolic region. 
A hyperbolic region $D\subset \mathbb{C}$
is {\em uniformly perfect} if the hyperbolic and quasihyperbolic densities are equivalent. 
Indeed, the following relation holds:
\begin{equation}\label{sec5-eq1}
\frac{1}{Q\delta_D(z)}\le \lambda_D(z)\le \frac{1}{\delta_D(z)} 
\end{equation}
for some constant $Q>0$ depending only on the domain $D$. 
Here, we call the constant $Q$ the {\em domain constant of uniformity}.
As noted in \cite{KL07}, the unit disk is a uniformly perfect hyperbolic region. In fact 
every convex hyperbolic region is uniformly perfect with the domain constant of uniformity $2$, 
see \cite{KM93,Min83}. However, there are domains which are not uniformly perfect. For instance,
the punctured disk $\mathbb{D}^*=\mathbb{D}\setminus\{0\}$ is not uniformly perfect. 
As an application to Theorem~\ref{sec3-thm1}, we finally 
collect the following results:  
\begin{theorem}
If $D$ is a uniformly perfect hyperbolic region in $\mathbb{C}$ and $f(z)$ is a conformal mapping in $D$
then 
$$\frac{1}{4}\lambda_D(z)\le \frac{|f'(z)|}{\delta_{f(D)}(f(z))}\le 4Q\,\lambda_D(z)
$$ 
for all $z\in D$. 
\end{theorem}
\begin{proof}
For the proof we use Theorem~\ref{sec3-thm1} and the lower bound from (\ref{sec5-eq1}).
\end{proof}

In this setting, it is evident to find a result of type \cite[Theorem~1]{Osg82}.
\begin{theorem}
If $D$ is a uniformly perfect hyperbolic region in $\mathbb{C}$ and $f(z)$ is a 
conformal mapping in $D$ then 
$$\Big|\frac{f''(z)}{f'(z)}\Big|\le 4Q\,\lambda_D(z)
$$ 
for all $z\in D$. 
\end{theorem}

In particular, in convex hyperbolic regions we get the following consequences:

\begin{corollary}
If $D$ is a convex hyperbolic region in $\mathbb{C}$ and $f(z)$ is a conformal mapping in $D$
then 
$$\frac{1}{8}\lambda_D(z)\le \frac{|f'(z)|}{\delta_{f(D)}(f(z))}\le 8\,\lambda_D(z)
$$ 
for all $z\in D$.
\end{corollary}

\begin{corollary}
If $D$ is a convex hyperbolic region in $\mathbb{C}$ and $f(z)$ is a conformal mapping in $D$
then 
$$\Big|\frac{f''(z)}{f'(z)}\Big|\le 8\,\lambda_D(z)
$$ 
for all $z\in D$.
\end{corollary}

%
%
%

\vskip 1cm
\noindent
{\bf Acknowledgement.} I thank the referee for his/her useful comments made in the earlier 
version of this manuscript.

\end{document}